\documentclass{article}

\usepackage{amsmath}
\usepackage{amsfonts}
\usepackage{amsthm}
\usepackage{amssymb}
\usepackage{amsbsy}
\usepackage{verbatim}
\usepackage{color}
\usepackage{mathrsfs}
\usepackage[hmargin = 1in,vmargin=1in]{geometry}
\usepackage[parfill]{parskip}

\newtheorem{thm}{Theorem}[section]
\newtheorem{cor}[thm]{Corollary}
\newtheorem{lem}[thm]{Lemma}
\theoremstyle{definition}
\newtheorem{defn}[thm]{Definition}
\newtheorem{prop}[thm]{Proposition}
\theoremstyle{remark}
\newtheorem{rmk}{Remark}

\def\adots{\mathinner{\mkern1mu\raise1pt\vbox{\kern7pt\hbox{.}}\mkern2mu
   \raise4pt\hbox{.}\mkern2mu\raise7pt\hbox{.}\mkern1mu}}

\newcommand{\mbf}[1]{\boldsymbol{\mathbf{#1}}}

\def\bA{{\mathbf A}}
\def\bB{{\mathbf B}}
\def\bC{{\mathbf C}}
\def\bD{{\mathbf D}}
\def\bI{{\mathbf I}}
\def\bL{{\mathbf L}}
\def\bQ{{\mathbf Q}}

\def\bP{{\mathbf P}}

\def\bU{{\mathbf U}}
\def\bV{{\mathbf V}}
\def\bW{{\mathbf W}}
\def\by{{\mathbf y}}

\def\C{{\mathbb C}}
\def\Z{{\mathbb Z}}

\def\Blt{{\mathscr{B}(\ell^2)}}

\def\Image{{\operatorname{Im}\,}}

\title{Localization of Matrix Factorizations}
\author{Ilya Krishtal\thanks{Department of Mathematics
Northern Illinois University, DeKalb, IL }, Thomas Strohmer\thanks{Department of Mathematics, University of California
    at Davis, Davis CA}, and Tim Wertz\thanks{Department of Mathematics, University of California
    at Davis, Davis CA}}

\begin{document}
\maketitle

\begin{abstract}
Matrices with off-diagonal decay appear in a variety of fields in
mathematics and in numerous applications, such as signal processing, statistics, communications engineering,  condensed matter physics, and quantum chemistry.
Numerical algorithms dealing with such matrices often take advantage (implicitly or explicitly) of the 
empirical observation that this off-diagonal decay property seems to be preserved when computing various useful
matrix factorizations, such as the Cholesky factorization or the QR-factorization.  
There is a fairly extensive theory
describing when the {\em inverse} of a matrix inherits the localization properties of the original matrix. 
Yet, except for the special case of band matrices, surprisingly very little theory exists that would establish 
similar results for  matrix factorizations.
We will derive a comprehensive framework to rigorously answer the question
when and under which conditions the matrix factors inherit the localization
of the original matrix for  such fundamental matrix factorizations  as the LU-, QR-, Cholesky,  and Polar factorization.  
\end{abstract}

\section{Introduction}
\label{s:intro}

Matrices with off-diagonal decay appear in numerous areas of mathematics including PDEs, numerical analysis,
pseudo-differential operator theory, and applied harmonic analysis. They also play a prominent role in applications, 
including signal processing, statistics, condensed matter physics, quantum chemistry, and communications
engineering. For instance, such matrices arise naturally from the fact that, in many
systems, perturbations are localized. That is, local disturbances are not felt globally. One way to study such
matrices is via various non-commutative generalizations of Wiener's Lemma, of which \cite{Gro10} provides an
overview. However in a number of applications one is not interested in just the decay properties of the 
matrix describing the system and its inverse, but also in the decay properties of various fundamental factorizations 
of that matrix. We give a few such examples below.

In quantum chemistry, the so-called \emph{density matrix} is a vital tool to determine the electronic structure of (possibly large) molecules~\cite{Mar2004,LeBris2005}.   Many of the important values in electronic structure theory can be obtained as functionals of the density matrix. The typical basis vectors used to discretize the Hamiltonian give rise to a matrix which, while not sparse, exhibits off-diagonal decay.
For non-metallic systems the off-diagonal decay of the density matrix is exponential, while for metallic systems
the decay is algebraic~\cite{Goe98,Benzi12}. This decay is exploited to devise
fast numerical algorithms for electronic structure calculations. Many of these algorithms implicitly assume that the (inverses of the) Cholesky factors of the density matrix inherit the decay properties
of the density matrix. Benzi, Boito, and Razouk recently showed that this assumption is indeed justified in case of
exponential decay~\cite{Benzi12}. However, for other types of decay, such as algebraic decay, it has been
an open problem whether the Cholesky factors do indeed possess the same off-diagonal decay as the density matrix.

In astronomy, the detection and identification of signals must be done in the presence of cosmic microwave background radiation. This is a specific instance of detecting sparse, and possibly faint, signals in the presence of noise. In \cite{HallJin09}, methods are developed to detect signals in the case that the noise is correlated. Such methods work within bounds which depend on the sparsity and strength of the signal, and require off-diagonal decay of the correlation matrix. Essential to this method is that the Cholesky factors of the correlation matrix exhibit the same form of decay.

In many applications one is interested in constructing an orthonormal system consisting of vectors that 
are well localized, e.g. see~\cite{CM06,Str05,GNG08} and \cite[Chapter 9]{mallat} for some examples.
A standard approach is to apply the Gram-Schmidt algorithm to a set of vectors, which is of course nothing else
than computing the QR-factorization of a matrix $\bA$, where $\bA$ contains the initial set of vectors as its
columns. When the matrix $\bA$ is localized, then it would be desirable to know (and often it is tacitly assumed,
without any proof) that the orthonormal vectors obtained via Gram-Schmidt inherit these localization properties. 
In numerical analysis, the off-diagonal decay of the Cholesky factors or of matrix functions such as the matrix
exponential, can be exploited to construct efficient preconditioners~\cite{BG99,Iserles00,BR07}. 
Current theory only covers the case where the matrix is banded or has exponential off-diagonal decay.

Furthermore, various signal processing algorithms for wireless communications involve the QR-factorization of a 
large matrix that represents the wireless communication channel~\cite{SHDK11} as an efficient form of ``precoding''.  
Moreover, the LU-factorization and the
Cholesky factorization plays a prominent role in signal processing, filter design, and digital 
communications~\cite{KSH00}, often in connection with the concept of causality. In all these applications, 
knowledge about the off-diagonal decay behavior of the 
matrix factors can greatly help to reduce computational complexity and mitigate truncation errors, 
as well as simplify hardware design.

The results presented in this paper can be seen as fundamental extensions of noncommutative versions of the 
famous Wiener's Lemma. 
Recall that these noncommutative generalizations of Wiener's Lemma state that under certain conditions the
inverse $\mbf{A}^{-1}$ of a matrix $\mbf{A}$ will indeed inherit the off-diagonal decay properties of $\mbf{A}$,
see e.g.~\cite{jaffard90,Bas90,B96,Kur99,GL03,Sun07,Gro10}.
Informally, this {\em Wiener property}  can be stated as ``{\em if $\mbf{A}$ is localized, then so is $\mbf{A}^{-1}$}''.
It has been an open problem whether and under which conditions the Wiener property extends to matrix
factorizations. The results in this paper provide affirmative answers.

We will show that for a wide range of off-diagonal decay, localized matrices give rise to 
LU-, Cholesky, QR-, and polar factorizations, whose factors inherit this localization. 
Below we state two examples of the type of results we prove in this paper. In both theorems the algebra $\mathscr{A}$ 
represents a Banach algebra that describes localized matrices in  the form of some off-diagonal decay.
To give a concrete example, we might assume that $\bA \in \mathscr{A}$ satisfies polynomial off-diagonal decay, i.e.,
there exists a parameter $s>1$ and a constant $C>0$ such that the entries $a_{jk}$ of $\bA$ obey
$$|a_{jk}| \leq C(1+|j-k|)^{-s}, \qquad \forall j,k \in \Z.$$
The precise statements of the theorems with the accompanying definitions are given in the following sections. 

\begin{thm}
Let $\mathscr{B}_c$ be the closure of the algebra of band matrices with respect to the operator norm. Let
$\mathscr{A} \subset \mathscr{B}_c$ be an inverse-closed sub-algebra satisfying certain technical conditions, given
in Section 3. Suppose $\mbf{A} \in \mathscr{A}$. Then, if $\mbf{A} =\mbf{LU}$ in $\mathscr{B}_c$, we have $\mbf{L,U} \in \mathscr{A}$. 
\end{thm}

\begin{thm}
Let $\mathscr{A}$ be one of the decay algebras specified in Section 4, such as the algebra with polynomial off-diagonal decay discussed above. Assume that $\mbf{A} \in \mathscr{A}$. Then, if $\mbf{A} = \mbf{QR}$, we have $\mbf{Q,R} \in \mathscr{A}$.
\end{thm}

The LU-factorization turns out to be the most challenging one for which to prove that the matrix factors exhibit
the same localization properties as the original matrix. 
We present two quite different methods for obtaining such localization results for the LU-factorization, see
Sections~\ref{s:algebraic} and~\ref{s:schur}.
The first method is more abstract and algebraic, while the second is more concrete and computational. The first method 
has the advantage of being applicable to a broad array of types of decay. To prove our theorem, we only need to 
assume that the matrix belongs to some inverse closed Banach algebra, and that it possesses an LU-factorization in 
a very weak decay algebra. The second method requires us to consider each form of decay separately, but it gives 
a better quantitative understanding of the decay properties of the factors. Moreover, we need to assume only that 
the matrix possesses an LU-factorization, but nothing about the factors themselves. In Section~\ref{s:cholesky}
we  prove similar localization results for the QR-factorization, the Cholesky factorization, and the polar
factorization.  In fact, for the latter the result is 
quickly obtained from basic facts and does not actually require the main machinery developed in this paper. 
Finally, localization of matrix functions, such as the matrix exponential, is the topic
of Section~\ref{s:functional}.

Prior work on the topic of localized matrix factorizations starts with the seminal work
of Wiener on spectral factorization, see e.g.~\cite{Wie32}. His research has led to a plethora of
extensions, generalizations, and refinements, which we cannot possibly review here.
This work, however, is firmly rooted in the commutative setting, as the operators under consideration
correspond to convolution operators (including even those papers dealing with matrix-valued convolution
operators).
To our knowledge the first result that truly addresses localization of matrix factorizations in
the noncommutative setting is due to Gohberg, Kasshoek, and Woerdeman. In~\cite{GKW89}, the authors
show if a positive definite matrix is in the so-called non-stationary Wiener algebra ${\cal W}$
(in current terminology, ${\cal W}$ is the unweighted Baskakov-Gohberg-Sj\"ostrand algebra,
see Definition~\ref{weights}, item~4), then the Cholesky factors belong to the same algebra.
Our research was definitely influenced by their paper as well as by its notable precursors
\cite{GL72, GF67}. We also gladly acknowledge inspiration
by the work of Benzi and coauthors, who have analyzed the localization of (inverses of) Cholesky
factors for the case of band matrices~\cite{Benzi12}.
Finally, Baskakov's work on the spectral theory of Banach modules \cite{B04} and especially
causal operators~\cite{BK05, BK04}
(in collaboration with one of the authors of this paper), has paved the way for one of the approaches
presented in this paper.

The reader may wonder why we have not yet mentioned the eigenvalue decomposition and the singular value 
decomposition. It is easy to see that localization of the matrix is in general not sufficient to ensure 
localization of its eigenvectors or singular vectors. Take for example a bi-infinite Laurent  matrix
(also called Toeplitz matrix), then we know that its (generalized) eigenvectors are given by the
complex exponentials, which have no decay whatsoever, regardless of how strong the off-diagonal decay of
the Laurent matrix may be.
Nevertheless, under certain conditions the eigenvectors or singular vectors of a matrix {\em do inherit}
the localization properties of a matrix. We will report on these results in a forthcoming paper.

\section{Preliminaries}
\label{s:preliminaries}

\subsection{Notation}

For a vector $\by = \{y_k\}_{k\in\Z} \in \ell^2(\Z)$ we define its projection $P_n \by$ to be
\begin{equation}
P_n \by = (\dots,0,y_{-n},y_{-n+1},\dots,y_{n-1},y_n,0,\dots).
\label{projn}
\end{equation}
The range of $P_n$ is written as $\Image P_n$. Since $\Image P_n$ is a subspace of $\ell^2(\Z)$ of dimension
$2n+1$, it can be identified with $\C^{2n+1}$.
Given a matrix $\bA = (a_{ij}), i,j\in\Z$ we denote $\bA^{(n)} = P_n \bA P_n$ restricted to $\Image P_n$, that is
\begin{equation}
\bA^{(n)}: \Image P_n \to \Image P_n \subset \ell^2(\Z).
\label{}
\end{equation}
By definition $\bA^{(n)}$ is an operator acting on $\Image P_n$, and we can interpret $\bA^{(n)}$ as
a finite $(2n+1)\times (2n+1)$ matrix with entries $\{a_{ij}\}_{|i|,|j| \leq n}$, acting on $\C^{2n+1}$. 

Recall that a matrix $\bA$ is lower triangular if $a_{ij} = 0$ for $i>j$,
and $\bA$ is upper triangular if $a_{ij} = 0$ for $i<j$.
We shall use the notations $\mathscr{L,L^*,D}$ to refer to the sub-algebras of $\mathscr{B}(\ell^2)$ consisting of 
lower-triangular, upper-triangular, and diagonal matrices, respectively. We shall denote by $\mathscr{L}_0$ 
and $\mathscr{L}_0^*$ the sub-algebras of strictly lower-triangular and strictly upper-triangular matrices, 
respectively.

\subsection{Matrix Factorizations}

For the convenience of the reader we briefly review the definitions of the matrix factorizations under consideration. 

\begin{defn}
\label{def:factorizations}
Let $\mathscr{A}$ be a Banach algebra of matrices in $\mathscr{B}(\ell^2)$ and assume that $\mbf{A}\in\mathscr{A}$. \\
1) 
We say that $\mbf{A}\in\mathscr{A}$ admits an 
\emph{LU-factorization} 
in $\mathscr{A}$ if $\mbf{A} = \mbf{LU}$, where $\mbf{L}, \mbf{L}^{-1} \in \mathscr{L}\cap\mathscr{A}$ and
$\mbf{U}, \mbf{U}^{-1}\in\mathscr{L}^{\ast}\cap\mathscr{A}$. \\ 
2) $\mbf{A}\in \mathscr{A}$ admits a \emph{QR-factorization} in $\mathscr{A}$ 
if $\mbf{A} = \mbf{QR}$, where $\mbf{R}, \mbf{R}^{-1} \in \mathscr{L^{\ast}}\cap\mathscr{A}$
and $\bQ\in \mathscr{A}$ is a unitary matrix. \\
3) Assume $\bA$ is hermitian positive definite. We say that $\mbf{A}\in \mathscr{A}$ admits a \emph{Cholesky factorization} 
in $\mathscr{A}$ if $\mbf{A} = \mbf{C C^{\ast}}$, where $\mbf{C}, \mbf{C}^{-1} \in \mathscr{L}\cap\mathscr{A}$
and the diagonal entries of $\mbf{C}$ are positive.\\
4) The {\em polar factorization} of $\mbf{A}\in \mathscr{A}$ is given by $\bA = \bU \bP$, where $\bU$ 
is a unitary matrix and $\bP$ is a positive-semidefinite Hermitian matrix, and $\bU,\bP \in \mathscr{A}$.
\end{defn}

\begin{rmk}\label{uniqrem}
For bi-infinite matrices the existence and uniqueness of some of these matrix factorizations 
is non-trivial, see~\cite{Arv75,CWS82,vdMRS96,Strang11} for more detailed discussions\footnote{The definitions
of the factorizations in Definition~\ref{def:factorizations} imply that $\bA$ is invertible in $\Blt$.
Some of the factorizations could be defined slightly more generally than we have don here.
For instance in the QR-factorization and the polar factorization
we could replace the unitary matrix by a partial isometry. Many of the results in our paper can be extended to hold
for these more general factorizations, but for clarity of presentation we prefer to work with the
factorizations as defined in Definition~\ref{def:factorizations}.}.
The existence of the QR-factorization and the Cholesky factorization follows for instance from the results
of Section~3 in Arveson's seminal paper~\cite{Arv75}. The existence of the LU-factorization is less clear
(unless we consider finite matrices), the interested reader may want to consult~\cite{ASW86} and the references
therein. A necessary condition that a matrix $\bA$ has an LU-factorization on $\ell^p, 1\le p \le \infty,$ is
that $\bA$ and all $\bA^{(n)}$ are uniformly invertible, i.e., $\sup_n
\{\|(\bA^{(n)})^{-1}\|_\Blt,\|\bA^{-1}\|_\Blt\} < \infty$,
see~\cite{BG70,ASW86}. However, this condition is not sufficient as a counter example in~\cite{ASW86} shows. 
We  observe that if $\mbf{LU} = \bar{\mbf{L}}\bar{\mbf{U}}$ are two different LU-factorizations then 
$\bar{\mbf{L}}^{-1} \mbf{L} = \bar{\mbf{U}} \mbf{U}^{-1}$ is a diagonal matrix, i.e.~the LU-factorization 
is unique up to multiplication by a diagonal matrix. In Section \ref{s:schur} we shall assume that the diagonal entries of $\bL$ are all equal to 1 to avoid ambiguity. For positive definite matrices the existence of such an LU-factorization follows immediately
from the existence of the Cholesky factorization. Indeed, one simply rescales $\bC$ via
multiplication by a diagonal matrix, i.e., if $\bA =\bC \bC^\ast$ then $\bA = \bL \bU$, where $\bL = \bC \bD^{-1}$ and the main diagonal of $\bD$ coincides with the main diagonal of $\bC$.
\end{rmk}

\subsection{Decay Algebras}

We first consider some typical matrix norms that express various forms of off-diagonal decay. In applications, one
might encounter such forms of decay in  signal and image processing, digital communication, quantum chemistry 
and quantum physics. Off-diagonal decay is  quantified by means of weight functions.

\begin{defn}\label{weights}
A non-negative function $v$ on $\mathbb{Z}$ is called an \emph{admissible weight} if it satisfies the following properties:
\begin{enumerate}
\item
$v$ is even and normalized so that $v(0) = 1$;
\item
$v$ is sub-multiplicative, i.e. $v(j+k) \leq v(j)v(k)$ for all $j,k \in \mathbb{Z}$;
\item
$v$ satisfies the \emph{Gelfand-Raikov-Shilov (GRS)} condition~\cite{gelfandraikov}:
$\lim\limits_{n \to \infty} v(nk)^{\frac{1}{n}} = 1$ 
for all $k \in \mathbb{Z}$.
\end{enumerate}
\end{defn}

The assumption that $v$ is even assures that the corresponding Banach algebra is closed under taking the adjoint. The GRS property is crucial for the inverse-closedness of the Banach algebra, as we will see below. The standard weight functions on $\mathbb{Z}$ are of the form
\[ v(k) = e^{a\cdot d(k)^b}(1+d(k))^s, \]
where $d(k)$ is a norm on $\mathbb{Z}$. Such a weight is sub-multiplicative when $a,s \geq 0$ and $0 \leq b \leq
1$, and satisfies the GRS condition if and only if $0 \leq b < 1$.

\begin{defn}\label{algebras}
We consider the following types of off-diagonal decay.
\begin{enumerate}
\item
The \emph{Jaffard Class} \cite{jaffard90}, denoted $\mathscr{A}_s$, is the collection of matrices $\mbf{A} = (a_{jk}), j,k \in \mathbb{Z}$ such that
\begin{equation}
\label{jaffard}
|a_{jk}| \leq C(1+|j-k|)^{-s}, 
\end{equation}
endowed with the norm
\[ \|\mbf{A}\|_{\mathscr{A}_s} := \sup_{j,k \in \mathbb{Z}}|a_{jk}(1+|j-k|)^s. \]
\item
More generally, let $v$ be an admissible weight such that $v^{-1} \in\ell^1({\mathbb{Z}})$ and $v^{-1}*v^{-1} \leq Cv^{-1}$. Then we denote by $\mathscr{A}_v$ the collection of matrices satisfying
\begin{equation}
\label{generalized_jaffard}
 |a_{jk}| \leq Cv^{-1}(j -k), 
\end{equation}
endowed with the norm
\[ \|\mbf{A}\|_{\mathscr{A}_v} := \sup_{j,k \in \mathbb{Z}}|a_{jk}|v(j-k). \]
\item {\em Schur-type algebras:}
Let $v$ be an admissible weight. Then we denote by $\mathscr{A}_v^1$ the collection of matrices $\mbf{A} = (a_{jk}),j,k \in \mathbb{Z}$ such that
\begin{equation}
 \sup_{j \in \mathbb{Z}}\sum_{k \in \mathbb{Z}}|a_{jk}|v(j-k) < \infty  \hspace{2pc} \mbox{and}  \hspace{2pc} \sup_{k \in \mathbb{Z}}\sum_{j \in \mathbb{Z}}|a_{jk}|v(j-k) < \infty, 
\label{schur}
\end{equation}
endowed with the norm
\[ \|\mbf{A}\|_{\mathscr{A}_v^1}:= \max\left\{ \sup_{j \in \mathbb{Z}}\sum_{k \in \mathbb{Z}}|a_{jk}|v(j-k) < \infty,\ \sup_{k \in \mathbb{Z}}\sum_{j \in \mathbb{Z}}|a_{jk}|v(j-k) < \infty \right\}. \]
\item
Let $v$ be an admissible weight. Then, the \emph{Gohberg-Baskakov-Sj\"{o}strand class}, denoted by $\mathscr{C}_v$, is the collection of matrices such that the norm
\[ \|\mbf{A}\|_{\mathscr{C}_v} := \sum_{j \in \mathbb{Z}} \sup_{k\in\mathbb{Z}}|a_{k,k-j}|v(j) = \inf_{\alpha \in \ell^1_v}\left\{ \|\alpha\|_{\ell^1_v} : |a_{jk}| \leq \alpha(j-k) \right\} \]
is finite.
\end{enumerate}
\end{defn}

If $\mathscr{A}$ is one of the Banach algebras defined in~\ref{algebras}, then any matrix $\bA\in\mathscr{A}$
is bounded on $\ell^p, 1\le p \le \infty$, see~\cite{GRS10}. The results derived in this paper hold for a variety
of other Banach algebras that describe off-diagonal decay, such as the ones in~\cite{Sun07}. But for clarity of presentation
we mainly focus on the Banach algebras introduced in Definition~\ref{algebras}.

We recall that a Banach algebra $\mathscr{A}$ is \emph{inverse-closed} in $\mathscr{B}(\ell^2(\mathbb{Z}))$ if for
every $\mbf{A} \in \mathscr{A}$ that is invertible on $\ell^2(\mathbb{Z})$ we have that $\mbf{A}^{-1} \in
\mathscr{A}$. The matrix algebras above are inverse-closed essentially when $v$ is an admissible weight. The
precise statement is slightly more involved, because we need to be a bit meticulous about the weights.
\begin{thm}
Let $v$ be an admissible weight.
\begin{enumerate}
\item
Assume that $v^{-1} \in \ell^1(\mathbb{Z})$ and $v^{-1}*v^{-1} \leq Cv^{-1}$. Then $\mathscr{A}_v$ is inverse-closed in $\mathscr{B}(\ell^2(\mathbb{Z}))$. In particular, $\mathscr{A}_s$ possesses this property if $s>1$.
\item
If $v(k)\geq C (1+|k|)^\delta$ for some $\delta > 0$, then $\mathscr{A}_v^1$ is inverse-closed in $\mathscr{B}(\ell^2(\mathbb{Z}))$.
\item
$\mathscr{C}_v$ is inverse closed in $\mathscr{B}(\ell^2(\mathbb{Z}))$ for arbitrary admissible weights.
\end{enumerate}
\end{thm}
While for $C^*$-(sub)algebras inverse-closedness is easy to prove and always true, it is highly non-trivial to
establish inverse-closedness of a Banach algebra. Inverse-closedness for $\mathscr{A}_s$ is due to Jaffard
\cite{jaffard90} and Baskakov \cite{Bas90, Bas97}. For $\mathscr{A}_v$ it was proved by Baskakov \cite{Bas97}, and
a different proof is given in \cite{GL04}. The result for $\mathscr{A}_v^1$ is proven in \cite{GL03}. Inverse-closedness for $\mathscr{C}_v$ with $v \equiv 1$ is due to Gohberg, Kaashoek, and Woerdeman \cite{GKW89}, and was rediscovered by Sj\"{o}strand \cite{Sjo95}. The case of arbitrary weights is due to Baskakov \cite{Bas90, Bas97}, but see also Kurbatov \cite{Kur99} and Blatov \cite{B96}.

\begin{rmk}
We note that for a singly-infinite matrix $\bA \in \mathscr{A}_v$ with admissible weight $v$, the existence of the LU-factorization
of $\bA$ can be derived from Theorem~2 in~\cite{ASW86}. This follows essentially from the fact that 
condition~\eqref{schur} together with the inverse-closedness of $\mathscr{A}_v$ imply the decay condition 
in Theorem~2 of~\cite{ASW86}. 
\end{rmk}

Furthermore we will make use of the following two classes of matrices.
The class of \emph{band} matrices, denoted $\mathscr{B}_b$, is the collection of matrices $\mbf{A} = (a_{jk}), j,k \in \mathbb{Z}$, such that there is some natural number $N=N(\mbf{A})$ with the property $a_{jk} = 0$ if $|j-k| > N$.
The class of matrices with \emph{exponentially decaying} diagonals, denoted $\mathscr{B}_{\gamma}$, is the
collection of matrices $\mbf{A} = (a_{jk}), j,k \in \mathbb{Z}$, such that  $|a_{jk}| \leq C\gamma^{|j-k|}$ for some constants $C = C(\mbf{A}) > 0$ and $\gamma = \gamma(\mbf{A}) \in (0,1)$.

All of the Banach algebras described in Definition~\ref{algebras} are contained in a larger Banach algebra, which 
we will denote
$\mathscr{B}_c$. In some sense, $\mathscr{B}_c$ defines the weakest sort of off-diagonal decay, but we will need to lay a little more groundwork before we are ready to give it a formal definition.

For $\theta \in \mathbb{T}$ and $x \in \ell^2(\mathbb{Z})$, the modulation representation $M: \mathbb{T}\to \ell^2(\mathbb{Z})$ is defined by
\[M(\theta)x(n) = \theta^n x(n).\]
Given any $\mbf{A}\in \mathscr{B}(\ell^2)$ we shall denote by $f_{\mbf{A}} \in L^\infty(\mathbb{T}, \mathscr{B}(\ell^2))$ the function
\begin{equation}\label{algincl}
f_{\mbf{A}}(\theta) = M(\theta)\mbf{A}M(\theta^{-1}), \ \theta \in\mathbb{T}.
\end{equation}
The Fourier series  of the operator $\mbf{A}$ is defined \cite{dL73, Bas97} 
as the Fourier series of the function $f_{\mbf{A}}(\theta) \sim \sum_k \theta^k \mbf{A}_k$. An easy computation shows that the operators $\mbf{A}_k$ are the diagonals of the matrix $\mbf{A}\in \mathscr{B}(\ell^2)$.

\begin{defn}
\label{Mcont}
The algebra of \emph{M-continuous} matrices, denoted $\mathscr{B}_c$, is the collection of all matrices such that the function $f_{\mbf{A}}$ defined above is a continuous map from $\mathbb{T}$ to $\mathscr{B}(\ell^2)$, i.e.~$f_{\mbf{A}} \in C(\mathbb{T}, \mathscr{B}(\ell^2))$. 
\end{defn}

In the following proposition we collect some useful known relationships between the above 
subalgebras, see~\cite{BK11, BK05} and references therein, as well as \cite{GK10, GrKl12}.

\begin{prop}\label{propprop} The following properties hold:
\begin{enumerate}
\item $\mathscr{B}_c(\ell^2) = \overline{\mathscr{B}_b(\ell^2)}$, (because of this property $\mathscr{B}_c(\ell^2)$ is sometimes also
referred to as the algebra of band-dominated operators, see e.g.~\cite{Lin2006});
\item $\mathscr{B}_b(\ell^2) \subseteq \{\mbf{A} \in \mathscr{B}(\ell^2)$: $f_{\mbf{A}}$ admits an extension to an entire function$\}$;
\item $\mathscr{B}_{\gamma}(\ell^2) = \{\mbf{A}\in \mathscr{B}(\ell^2)$: $f_{\mbf{A}}$ admits a holomorphic extension to an annulus
$\{1-\varepsilon < |z| < 1+\varepsilon\}$ for some $\varepsilon > 0$ depending on $\gamma\}$.
\item $\mathscr{B}_b(\ell^2) \subset \mathscr{B}_{\gamma}(\ell^2) \subset \mathscr{C}_1 \subset \mathscr{B}_c(\ell^2)$;
\item $\mathscr{L} \cap \mathscr{B}_c(\ell^2) = \{\mbf{A}\in \mathscr{B}(\ell^2)$: $f_{\mbf{A}}$ admits a holomorphic extension into the unit disc $\mathbb{D} = \{z\in\mathbb{C}: |z|< 1\}$ that is continuous in $\overline{\mathbb{D}}\}$;
\item $\mathscr{L}^* \cap \mathscr{B}_c(\ell^2) = \{\mbf{A}\in \mathscr{B}(\ell^2)$: $f_{\mbf{A}}$ admits a bounded holomorphic extension outside the unit disc $\mathbb{D}$ that is continuous in $\mathbb{C} \backslash{\mathbb{D}}\}$;
\item $\mathscr{L}_0$ and $\mathscr{L}_0^*$ are two-sided ideals in $\mathscr{L}$ and $\mathscr{L}^*$ respectively.
\end{enumerate}
\end{prop}

By $\mathscr{A}$ we shall denote some Banach algebra of matrices in $\mathscr{B}(\ell^2)$. Typically we shall 
assume that 
\begin{equation}\label{incl}
\mathscr{B}_{\gamma}(\ell^2)\subset \mathscr{A}\subseteq  \mathscr{B}_c(\ell^2)\subset \mathscr{B}(\ell^2)
\end{equation} 
and
\begin{equation}\label{dominorm}
\|\mbf{A}\|_{\mathscr{B}(\ell^2)} \le c\|\mbf{A}\|_{\mathscr{A}},
\end{equation}
where the constant $c$ is independent of $\mbf{A}$. These conditions are satisfied for the Banach Algebras
of Definition~\ref{algebras}. Indeed, for condition~(\ref{incl}) this is obvious and for condition~\eqref{dominorm}
see~\cite{GRS10}.

\begin{defn}
We say that an algebra $\mathscr{A}$ is \emph{strongly decomposable} if there exists a bounded projection $\mathcal{P}:\mathscr{A}\to\mathscr{A}$ which maps
$\mathscr{A}$ onto $\mathscr{L}\cap\mathscr{A}$ parallel to $\mathscr{L}_0^*$. In this case, we let $\mathcal{Q} = I-\mathcal{P}\in B(\mathscr{A})$ be the projection onto  $\mathscr{L}_0^*\cap\mathscr{A}$ parallel to $\mathscr{L}$. 
\end{defn}
\begin{rmk}
It is easy to check that the Banach algebras
of Definition~\ref{algebras} are strongly decomposable. In most cases,
$\mathscr{B}\subseteq \mathscr{C}_1$ will imply strong decomposability.
\end{rmk}


In order to prove the main result in Section \ref{s:algebraic}, we need to embed the algebra $\mathscr{A}$ into the algebra of (uniformly) continuous operator valued functions
$C(\mathbb{T}, \mathscr{B}(\ell^2))$.  The above  definitions transfer to the realm of such functions in the following way.

\begin{defn}\cite{GL72}.
Let $\mathfrak{A}$ be a subalgebra of $C(\mathbb{T}, \mathscr{B}(\ell^2))$. Let $\mathfrak{A}^+$ be the set of all functions in 
$\mathfrak{A}$ that admit holomorphic extensions into $\mathbb{D}$ that are continuous in $\overline{\mathbb{D}}$. Similarly, let $\mathfrak{A}^-$ be the set of all functions in 
$\mathfrak{A}$ that admit bounded holomorphic extensions into $\mathbb{C}\backslash\overline{\mathbb{D}}$ that are continuous in $\mathbb{C}\backslash{\mathbb{D}}$. Let also 
$\mathfrak{A}^-_0 = \{f\in\mathfrak{A}^-: \ f(\infty) = 0\}.$ The algebra $\mathfrak{A}$ is \emph{decomposing}
if $\mathfrak{A} = \mathfrak{A}^+ \oplus \mathfrak{A}^-_0$. We shall denote by $\mathfrak{P}$ and $\mathfrak{Q}$ the projections in $B(\mathfrak{A})$ associated with this decomposition.
\end{defn}

\begin{defn}\cite{GL72}.
A function $f\in\mathfrak{A}$ admits a \emph{canonical factorization} if $f = f_\ell f_u$, where $f_\ell, f^{-1}_\ell \in\mathfrak{A}^+$ and $f_u,f^{-1}_u \in\mathfrak{A}^-$.
\end{defn}

\begin{rmk}\label{embedrmk}
We observe that
an algebra $\mathfrak{A}$ is  inverse closed in $C(\mathbb{T}, \mathscr{B}(\ell^2))$ if and only if the function $f^{-1}$ defined by $f^{-1}(\theta) = [f(\theta)]^{-1}$ is in $\mathfrak{A}$ whenever
$f\in \mathfrak{A}$ and $[f(\theta)]^{-1}\in \mathscr{B}(\ell^2)$ for all $\theta\in \mathbb{T}$. In particular, a subalgebra $\mathfrak A_{\mathscr{A}} = \{f_{\mbf{A}}$: $\mbf{A}\in \mathscr{A}\}$ is inverse closed in $C(\mathbb{T}, \mathscr{B}(\ell^2))$ if  $f_{\mbf{A}}$ is defined by
\eqref{algincl} and $\mathscr{A}$ is inverse closed in $\mathscr{B}_c(\ell^2)$.
We also note that the algebra $\mathfrak A_{\mathscr{A}}$ is decomposing if $\mathscr{A}$ is strongly
decomposable.
\end{rmk}

\section{Abstract Harmonic Analysis Approach}
\label{s:algebraic}

In this section we present our first theorem concerning the off-diagonal decay of the LU-factors.
Our abstract approach, which is based on some advanced harmonic analysis results, leads to a fairly short 
proof, albeit at the cost of more conceptual effort.

\begin{thm}
\label{mainthm1}
Let $\mathscr{A}$ be a strongly decomposable inverse closed subalgebra of $\mathscr{B}(\ell^2)$ that satisfies
conditions~\eqref{incl}  and \eqref{dominorm}. Assume also that $\mbf{A}\in \mathscr{A}$ admits an 
LU-factorization $\mbf{A}=\mbf{LU}$ in $\mathscr{B}_c(\ell^2)$. Then $\mbf{A}$ admits an 
LU-factorization in $\mathscr{A}$.
\end{thm}

The proof of this result relies heavily on the following two abstract results.

\begin{thm}\label{holex}\cite[Theorem 8.14]{BK05}.
Let $\mbf{A}\in \mathscr{L} \cap \mathscr{B}_c(\ell^2)$. The following are equivalent:
\begin{enumerate}
\item $\mbf{A}^{-1}\in \mathscr{L}$;
\item  $f_{\mbf{A}}(z)$ is invertible in
$\mathscr{B}(\ell^2)$ for all $z\in\overline{\mathbb D}$.
\end{enumerate}
\end{thm}


\begin{thm}\label{GL72thm}\cite[Theorem 1.1]{GL72}.
Assume that $\mathfrak{A}$ is a decomposing inverse closed subalgebra of $C(\mathbb{T}, \mathscr{B}(\ell^2))$ that satisfies the following two conditions:
\begin{enumerate}
\item For all $f \in\mathfrak{A}$
$$\max_{\theta \in\mathbb{T}}  \|f(\theta)\|_{\mathscr{B}(\ell^2)} \le c\|f\|_{\mathfrak{A}},$$
where $c > 0$ is a constant independent of $f$;
\item If $f \in C(\mathbb{T}, \mathscr{B}(\ell^2))$ admits a holomorphic extension to an annulus
$\{1-\varepsilon < |z| < 1+\varepsilon\}$ for some $\varepsilon > 0$ then $f\in\mathfrak{A}$ and the set of all such operator-valued functions is dense in $\mathfrak{A}$.
\end{enumerate}

Then any function $f \in\mathfrak{A}$ that satisfies
\begin{equation}
\max_{\theta\in\mathbb{T}} \|f(\theta) - I\|_{\mathscr{B}(\ell^2)} < 1
\end{equation}
admits a canonical factorization $f = f_\ell f_u$ and
\begin{equation}
f_\ell^{-1}(\theta) = f_\ell(\theta)^{-1} = I - \mathfrak P g(\theta) + \mathfrak P[g(\theta)\mathfrak P g(\theta)] -\ldots,
\end{equation}
\begin{equation}
f_u^{-1}(\theta) = f_u(\theta)^{-1} = I - \mathfrak Q g(\theta) + \mathfrak Q[(\mathfrak Q g(\theta))g(\theta)] -\ldots,
\end{equation}
where $g(\theta) = f(\theta) - I$, and the series converge in $\mathfrak{A}$.
\end{thm}

The relationships between various subalgebras mentioned in Proposition \ref{propprop} and Remark \ref{embedrmk} allow us to
apply the above result to the algebras of the form
$$ \mathfrak{A}_{\mathscr{A}} = \{f \in C(\mathbb{T}, \mathscr{B}(\ell^2)):f = f_{\mbf{A}}, \mbf{A}\in\mathscr{A}\}.$$ 
As a result, we obtain the following special case of Theorem \ref{GL72thm}.

\begin{thm}\label{GLdec}
Assume that $\mathscr{A}$ is a strongly decomposable inverse closed subalgebra of $\mathscr{B}(\ell^2)$ that satisfies
conditions \eqref{incl} and \eqref{dominorm}. Then any matrix $\mbf{A}\in\mathscr{A}$ that satisfies $\|\mbf{A}-\mbf{I}\|_{\mathscr{B}(\ell^2)} < 1$ 
admits an LU-factorization $\mbf{A} = \mbf{LU}$ in $\mathscr{A}$ such that
\begin{equation}\label{cpart1}
\mbf{L}^{-1} = \mbf{I} - \mathcal{P} \mbf{M} + \mathcal{P}[\mbf{M}\mathcal{P} \mbf{M}] - \mathcal{P}[\mbf{M}\mathcal{P}[\mbf{M}\mathcal{P} \mbf{M}]] +\ldots,
\end{equation}
\begin{equation}\label{acpart1}
\mbf{U}^{-1} = \mbf{I} - \mathcal{Q}  \mbf{M} + \mathcal{Q}[[\mathcal{Q} \mbf{M}] \mbf{M}] - \mathcal{Q}[\mathcal{Q}[[\mathcal{Q} \mbf{M}] \mbf{M}] \mbf{M}] +\ldots,
\end{equation}
where  $\mbf{M} = \mbf{A}- \mbf{I}$ and the series converge in $\mathscr{A}$. 
\end{thm}

\begin{proof}[Proof of Theorem~\ref{mainthm1}]
Let $\mbf{A} = \mbf{LU}$ and consider the holomorphic extensions $f_{\mbf{L}}(z) = \sum_k z^k\mbf{L}_k$, $z\in{\mathbb D}$, and $f_{\mbf{U}}(z) = \sum_k z^k\mbf{U}_k$, $z\in \mathbb{C}\backslash\overline{\mathbb D}$. Pick $\varepsilon \in (0,1)$ such that
\[\|[f_{\mbf{L}}(\varepsilon)]^{-1}\mbf{LU}[f_{\mbf{U}}(1/\varepsilon)]^{-1}-\mbf{I}\| < 1.\]
Such an $\varepsilon$ exists because $\mbf{L},\mbf{U} \in \mathscr{B}_c(\ell^2)$ and invertibility is stable under small perturbations. Hence, Theorem \ref{GLdec} implies that 
$\mbf{A}': = [f_{\mbf{L}}(\varepsilon)]^{-1}\mbf{LU}[f_{\mbf{U}}(1/\varepsilon)]^{-1}$ admits an LU-factorization 
$\mbf{A}' = \mbf{L}'\mbf{U}'$ in $\mathscr{A}$. On the other hand, since $\mbf{L}^{-1} \in\mathscr{L}$ and
$\mbf{U}^{-1}\in\mathscr{L}^*$, Theorem \ref{holex} implies $[f_{\mbf{L}}(\varepsilon)]^{-1}\in \mathscr{L}$ and
$[f_{\mbf{U}}(1/\varepsilon)]^{-1}\in\mathscr{L}^*$. Hence, $\mbf{A}'$ admits two LU-factorizations in $\mathscr{B}_c(\ell^2)$ and
Remark \ref{uniqrem} implies that there exists an invertible diagonal matrix $\mbf{D}\in\mathscr{D}$ such that
\[(\mbf{L}')^{-1}[f_{\mbf{L}}(\varepsilon)]^{-1}\mbf{L} =  \mbf{U}' f_{\mbf{U}}(1/\varepsilon)\mbf{U}^{-1} = \mbf{D}.\]
Moreover, since $f_{\mbf{L}}(\varepsilon)$, 
$f_{\mbf{U}}(1/\varepsilon)\in \mathscr{B}_{\gamma}(\ell^2)$, 
 we have
$\mbf{L} = f_{\mbf{L}}(\varepsilon) \mbf{L}' \mbf{D} \in \mathscr{A}$ and $\mbf{U} = \mbf{D}^{-1} \mbf{U}' f_{\mbf{U}}(1/\varepsilon)\in\mathscr{A}$.
\end{proof}

\begin{rmk}
From the proofs above it is not hard to see that an analogous result holds in a much more general setting 
than just for matrix algebras. However, this general result is beyond the scope of this paper as it would require
introduction of too much heavy machinery. We cite \cite{BK05, BK11} 
which describe the setup for the general result.
\end{rmk}

\begin{rmk}
If in Theorem~\ref{mainthm1}, we additionally required that $\|\bI - \bA\|_{\mathscr{A}} < 1$, then the result
would follow immediately from the observation that the series in~\eqref{cpart1} and~\eqref{acpart1} would
converge in $\mathscr{A}$ under this additional assumption. However, the condition $\|\bI - \bA\|_{\mathscr{A}} < 1$ is
rather restrictive, as it {\em cannot} be enforced by a simple rescaling of $\bA$,
even not if we assumed $\bA$ to be positive definite. Indeed, it is highly non-trivial that we can
switch from the ${\mathscr A}$-norm to the operator norm in the condition $\|\bI-\bA\| < 1$
in Theorem~\ref{GLdec}. Note furthermore that for a non-positive definite matrix $\bA$ even the condition
$\|\bI-\bA\|_{\Blt} < 1$ cannot be enforced by simple rescaling. Instead we need to resort to
a smart choice of ``preconditioning'' by $[f_{\bL}(\varepsilon)]^{-1}$.  Establishing the existence of such a preconditioner with all 
the right properties involves some advanced results from abstract harmonic analysis.
\end{rmk}

\section{Linear Algebra Approach}
\label{s:schur}

As pointed out in Section~\ref{s:preliminaries}, if the LU-factorization of a matrix exists, then it is unique up to multiplication by an invertible diagonal matrix.
In this section we assume that $\bA$ has an LU-factorization $\bA = \bL \bU$, and to avoid ambiguity we stipulate
that the diagonal entries of $\bL$ are equal to 1.

In what follows it will be useful to write our matrices in block form, i.e.
\[ \mbf{A} = \left( \begin{array}{ll} 
\mbf{A}_{11} & \mbf{A}_{12} \\ 
\mbf{A}_{21} & \mbf{A}_{22} 
\end{array}\right) = 
\left( \begin{array}{ll} 
\mbf{L}_{11} & 0 \\ 
\mbf{L}_{21} & \mbf{L}_{22} 
\end{array}\right)
\left( \begin{array}{ll} 
\mbf{U}_{11} & \mbf{U}_{12} \\ 
0 & \mbf{U}_{22} 
\end{array}\right), \]
with $\bA_{11} = (a_{ij})_{i<0,j<0}$, $\bA_{12} = (a_{ij})_{i<0,j\ge 0}$,
$\bA_{21} = (a_{ij})_{i\ge 0,j<0}$, and $\bA_{22} = (a_{ij})_{i \ge 0,j\ge 0}$.
Analogously,
\[ \mbf{A}^{-1} = \mbf{B} = \left( \begin{array}{ll} 
\mbf{B}_{11} & \mbf{B}_{12} \\ 
\mbf{B}_{21} & \mbf{B}_{22} 
\end{array}\right) = \left( \begin{array}{ll} 
\mbf{\Omega}_{11} & \mbf{\Omega}_{12} \\ 
0 & \mbf{\Omega}_{22} 
\end{array}\right)
\left( \begin{array}{ll} 
\mbf{\Lambda}_{11} & 0 \\ 
\mbf{\Lambda}_{21} & \mbf{\Lambda}_{22} 
\end{array}\right), \]
where each of the blocks is only singly infinite.
It is sometimes convenient to consider the individual blocks of $\bA$ and $\bA^{-1}$
as operators acting on $\ell^2(\Z)$.
Thus (with slight abuse of notation) we can think of, say, $\bA_{11}$ as
$$
\begin{bmatrix}
\bA_{11} & 0 \\
0   & 0
\end{bmatrix}.
$$
We will also make use of the relations 
$\mbf{\Lambda}_{22} = \mbf{U}_{22}\mbf{B}_{22}$, $\mbf{L}_{11} = \mbf{A}_{11}\mbf{\Omega}_{11}$,
$\mbf{\Omega}_{22}^{-1} = \bU_{22}$ and $\mbf{\Lambda}_{22}^{-1} = \bL_{22}$, and note that these
relations imply invertibility of $\bB_{22}$ and $\bA_{11}$.
Furthermore, for a singly-infinite matrix  $\mbf{M} =(m_{jk})_{j,k =-1}^{-\infty}$ and $n \in \mathbb{N}$, we use the notation $\mbf{M}^{(n)}$ to refer to the sub-matrix consisting of those entries $m_{jk}$ such that 
$j,k=-1,\dots,-n$.  We will apply this notation to ``upper-left'' block matrices such as $\bA_{11}$ and $\bB_{11}$. Similarly, for a singly-infinite matrix  $\mbf{M} =(m_{jk})_{j,k =0}^{\infty}$ and $n \in \mathbb{N}$, 
we use the notation $\mbf{M}^{(n)}$ to refer to the sub-matrix consisting of those entries $m_{jk}$ such that  $j,k=0,\dots,n$.  We will apply this notation to ``lower-right' block matrices such as $\bA_{22}$ and $\bB_{22}$.  We use the notation $\mbf{m}_{n-1}$ to refer to the column vector consisting of the entries $m_{k,-n}$ such that $-n < k \leq 0$ and  $\mbf{m}^*_{n-1}$ to refer to the row vector consisting of the entries $m_{n,k}$ such that $0 \leq k < n$.

\begin{thm}\label{main2}
Let $\mathscr{A}$ be one of the algebras $\mathscr{A}_v, \mathscr{A}_{v}^1, \mathscr{C}_v$, introduced
in Definition~\ref{algebras}, and let $\mbf{A} \in \mathscr{A}$.
If $\bA = \bL \bU$ is the LU-factorization of $\bA$ in $\Blt$, then $\mbf{L},\mbf{U} \in \mathscr{A}$.
\end{thm}

The proof of Theorem~\ref{main2} is quite different from the approach in the previous section, it is more concrete
(and longer)  and implicitly utilizes the Schur complement. 
Before we proceed to the proof of this theorem we need some preparation.

We first prove some estimates on the entries. We shall deal with the lower-right blocks (e.g. $\mbf{\Lambda}_{22}$)  and the upper-left blocks (e.g. $\mbf{L}_{11}$) separately. The following lemma is adapted and expanded from \cite{HallJin09}.
\begin{lem}\label{le:bounds}
1) Let  $\mbf{B}_{22}^{-1} = (\beta_{jk})$. Then
\begin{equation}
\label{lambda_bound}
|\lambda_{kj}| \leq  \left| \mbf{\beta}^*_{k-1}\mbf{B}_{22}^{(k-1)}(j) \right| 
\end{equation}
for $0 \leq j \leq k-1$.\\
2) Let $\mbf{A}_{11}^{-1} = (\alpha_{jk})$. Then
\begin{equation}
\label{l_bound}
|\ell_{jk}| \leq  \left| \mbf{A}_{11}^{(k-1)}\mbf{\alpha}^*_{k-1}(j) \right|
\end{equation}
for $1-k \leq j \leq 0$.
\end{lem}

\begin{proof}
We first write
\[ \mbf{L}_{22} = \left( \begin{array}{lll} 
\mbf{L}_{22}^{(k-1)} & 0 & 0 \\ 
\mbf{\ell}^*_{k-1} & \ell_{kk} & 0 \\ 
\mbf{M} & \mbf{v} & \mbf{D} \end{array}\right), \]
from which we see
\[ \mbf{L}_{22}^{-1} = \mbf{\Lambda}_{22} = \left( \begin{array}{lll} 
\mbf{\Lambda}_{22}^{(k-1)} & 0 & 0 \\ 
\mbf{\lambda}^*_{k-1} & \lambda_{kk} & 0 \\ 
\mbf{N} & \mbf{w} & \mbf{D}^{-1} 
\end{array}\right). \]
By observing that
$$
\begin{bmatrix}
 \mbf{\ell}^*_{k-1} & \ell_{kk} & 0 
\end{bmatrix}
\begin{bmatrix}
\mbf{\Lambda}_{22}^{(k-1)} \\
\mbf{\lambda}^*_{k-1} \\
\mbf{N} 
\end{bmatrix} = 0,$$
we obtain
\begin{equation}
\label{Lambda_estimate}
 \mbf{\lambda}^*_{k-1} = -\ell_{kk}^{-1}\mbf{\ell}^*_{k-1}\mbf{\Lambda}_{22}^{(k-1)}. 
\end{equation}
Since $\mbf{\Lambda}_{22} = \mbf{U}_{22}\mbf{B}_{22}$, we see that 
$\mbf{L}_{22} = \mbf{B}_{22}^{-1}\mbf{\Omega}_{22}$. Hence, 
\begin{align} 
\mbf{\ell}^*_{k-1}\mbf{\Lambda}_{22}^{(k-1)} & = (\ell_{k1},\dots,\ell_{k,k-1})\mbf{\Lambda}_{22}^{(k-1)} =
\left(\sum_{i =1}^{k-1} \beta_{ki}\omega_{i1},\dots,\sum_{i = 1}^{k-1}\beta_{ki}\omega_{i,k-1}
\right)\mbf{\Lambda}_{22}^{(k-1)} \notag \\ 
&= \left(\sum_{j=1}^{k-1}\sum_{i =1}^{k-1} \beta_{ki}\omega_{ij}\lambda_{j1},\dots,\sum_{j=1}^{k-1}\sum_{i =
1}^{k-1}\beta_{ki}\omega_{ij}\lambda_{j,k-1}  \right) \notag \\
& = \left(\sum_{i=1}^{k-1}\beta_{ki}\sum_{j =1}^{k-1}
\omega_{ij}\lambda_{j1},\dots,\sum_{i=1}^{k-1}\beta_{ki}\sum_{j = 1}^{k-1}\omega_{ij}\lambda_{j,k-1}  \right)
\notag \\
& = \left(\sum_{i =1}^{k-1} \beta_{ki}b_{i1},\dots,\sum_{i = 1}^{k-1}\beta_{ki}b_{i,k-1}  \right) =
\mbf{\beta}^*_{k-1}\mbf{B}_{22}^{(k-1)}. \label{B_estimate}
\end{align}
Combining~\eqref{B_estimate} with~\eqref{Lambda_estimate} and noting that $\lambda_{kk} = 1$ (since 
the diagonal entries of $\bL$ are 1), we arrive at~\eqref{lambda_bound}.

The proof for~\eqref{l_bound} is similar.
\end{proof}

\begin{lem}\label{21}
Assume $\bA$ has the LU-factorization $\mbf{A} = \mbf{L}\mbf{U}$ in $\Blt$ and let $\mathscr{A}$ be one of the
algebras in Definition~\ref{algebras}. Suppose $\mbf{A} \in \mathscr{A}$ and the elements of $\mbf{L}$ in the
blocks $\mbf{L}_{11}$ satisfy the appropriate decay condition. Then the elements in the block $\mbf{L}_{21}$ also satisfy that decay condition.
\end{lem}

\begin{proof}
It is convenient for the proof of this lemma to consider the individual blocks of $\bA$ and $\bA^{-1}$
as operators acting on $\ell^2(\Z)$ as indicated in the beginning of this section.
Clearly, $\mbf{L}_{21} = \mbf{A}_{21}\mbf{\Omega}_{11}$, and due to the upper-triangular structure of
$\mbf{U}$ we also have that $\mbf{\Omega}_{11} = \mbf{U}_{11}^{-1}$. Moreover $\mbf{A}_{11} =
\mbf{L}_{11}\mbf{U}_{11}$. Hence, since the algebra $\mathscr{A}$ to which $\mbf{A}_{11}$ and $\mbf{L}_{11}$
belong is inverse-closed, we have that $\mbf{U}_{11} \in \mathscr{A}$. Thus $\mbf{\Omega}_{11} \in \mathscr{A}$
and therefore $\mbf{L}_{21} \in \mathscr{A}$.
\end{proof}

Now, to prove Theorem \ref{main2}, we only need to show that if $\mbf{A} \in \mathscr{A}$, then $\mbf{L}_{11}$ and
$\mbf{L}_{22}$ are also in $\mathscr{A}$.
\begin{proof}[Proof of Theorem~\ref{main2}]
\begin{enumerate}
\item
Suppose that $\mbf{A} \in \mathscr{A}_v$. Then, by~\eqref{lambda_bound} of Lemma \ref{le:bounds}, we have
\[ |{\lambda}_{kj}| \leq \left| \mbf{\beta}^*_{k-1}\mbf{B}_{22}^{(k-1)}(j) \right|, \]
for $k,j \geq 0$. Since both $\mbf{B}_{22}$ and $\mbf{B}_{22}^{-1}$ have the appropriate decay property, we have
\[ |{\lambda}_{kj}| \leq  \left| \mbf{\beta}^*_{k-1}\mbf{B}_{22}^{(k-1)}(j) \right| \leq \sum_{i =1}^{k-1} |\beta_{ki}||b_{ij}| \leq  \sum_{i=1}^{k-1} \frac{1}{v(k-i)}\frac{1}{v(i - j)} \leq \frac{1}{v(k-j)},\]
where the last inequality holds by the fact that $v$ is sub-convolutive. 
Thus we have shown that the entries of ${\mbf{\Lambda}}_{22}$ satisfy the appropriate decay estimate. Then, since $\mathscr{A}_v$ is inverse closed, we have that 
$\mbf{L}_{22}$ satisfies the decay estimate. 

By \eqref{l_bound} of Lemma \ref{le:bounds} we have
\[ |\ell_{jk}| \leq  \left| \mbf{A}_{11}^{(k-1)}\mbf{\alpha}_{k-1}(j) \right|, \]
for $j,k \leq -1$. Since both $\mbf{A}_{11}$ and $\mbf{A}_{11}^{-1}$ have the appropriate decay, we have
\[ |\ell_{jk}| \leq  \left| \mbf{A}_{11}^{(k-1)}\mbf{\alpha}_{k-1}(j) \right| \leq\sum_{i =-1}^{1-k} |a_{ji}||\alpha_{ik}| \leq  \sum_{i=1}^{k-1} \frac{1}{v(j-i)}\frac{1}{v(i - k)} \leq \frac{1}{v(j-k)} . \]
Hence $\mbf{L}_{11}$ satisfies the decay estimate. Then, by Lemma \ref{21}, $\mbf{L}_{21}$ satisfies the decay
condition. So all three blocks of $\mbf{L}$ have the appropriate decay. Therefore, $\mbf{L} \in \mathscr{A}_v$.
Then, since $\mathscr{A}_v$ is inverse-closed, $\mbf{\Lambda} \in \mathscr{A}_v$. Thus, since $\mbf{U} = \mbf{\Lambda}\mbf{A}$, we have that $\mbf{U} \in \mathscr{A}_v$.
\item
Suppose that $\mbf{A} \in \mathscr{A}_v^1$.  Then, by \eqref{lambda_bound}, we have
\[ |{\lambda}_{kj}| \leq  \left| \mbf{\beta}^*_{k-1}\mbf{B}_{22}^{(k-1)}(j) \right|, \]
for $k,j \geq 0$. Since both $\mbf{B}_{22}$ and $\mbf{B}_{22}^{-1}$ have the appropriate decay property, we have
\begin{align*}
\sup_{k \geq 0}\sum_{j \geq 0}|{\lambda}_{kj}|v(k-j) & \leq \sup_{k \geq 0}\sum_{j \geq 0} \left|
\mbf{\beta}^*_{k-1}{\mbf{B}}_{22}^{(k-1)}(j) \right|v(k-j) \\ &\leq \sup_{k \geq 0}\sum_{j \geq 0}\sum_{i =1}^{k-1} |\beta_{ki}||b_{ij}|v(k-j) \\
& \leq \sup_{k \geq 0}\sum_{j \geq 0}\sum_{i \geq 0} |\beta_{ki}||b_{ij}|v(k-i)v(i-j) \\ & \leq \sup_{k \geq 0}\sum_{i \geq 0} |\beta_{ki}|v(k-i)\sum_{j \geq 0}|b_{ij}|v(i-j) \\
& \leq \sup_{k \geq 0}\sum_{i \geq 0} |\beta_{ki}|v(k-i) \sup_{i \geq 0}\sum_{j \geq 0}|b_{ij}|v(i-j) < \infty.
\end{align*}
Note that the same argument applies if we take the supremum with respect to $j$ and the sum with respect to $k$
instead. Hence we have shown that the entries of ${\mbf{\Lambda}}_{22}$ satisfy the appropriate decay estimate.
Then, since $\mathscr{A}_v^1$ is inverse closed, we have that $\mbf{L}_{22}$ satisfies the decay estimate. 

Now,  by  \eqref{l_bound} we have
\[ |\ell_{jk}| \leq  \left| \mbf{A}_{11}^{(k-1)}\mbf{\alpha}_{k-1}(j) \right|, \]
for $j,k \leq -1$. Since both $\mbf{A}_{11}$ and $\mbf{A}_{11}^{-1}$ have the appropriate decay, we have
\begin{align*}
\sup_{k \leq -1}\sum_{j \leq -1}|{\ell}_{jk}|v(j-k) & \leq \sup_{k \leq -1}\sum_{j \leq -1}\left|
\mbf{A}_{11}^{(k-1)}\mbf{\alpha}_{k-1}(j) \right|v(j-k) \\
& \leq \sup_{k \leq -1}\sum_{j \leq -1}\sum_{i =-1}^{1-k} |a_{ji}||\alpha_{ik}|v(j-k) \\
& \leq \sup_{k \leq -1}\sum_{j \leq -1}\sum_{i \leq -1} |a_{ji}||\alpha_{ik}|v(j-i)v(i-k) \\
& \leq \sup_{k \leq -1}\sum_{i \leq -1} |\alpha_{ik}|v(i-k)\sum_{j \leq -1}|a_{ji}|v(j-i) \\
& \leq \sup_{k \leq -1}\sum_{i \leq -1} |\alpha_{ik}|v(i-k) \sup_{i \leq -1}\sum_{j \leq -1}|a_{ji}|v(j-i) < \infty.
\end{align*}
Again, the same argument applies if we take the supremum with respect to $j$ and the sum with respect to $k$
instead. Hence $\mbf{L}_{11}$ satisfies the decay estimate. Then, by Lemma \ref{21}, $\mbf{L}_{21}$ satisfies the
decay condition. So all three blocks of $\mbf{L}$ have the appropriate decay. Therefore, $\mbf{L} \in
\mathscr{A}_v^1$. Then, since $\mathscr{A}_v^1$ is inverse-closed, $\mbf{\Lambda} \in \mathscr{A}_v$. Hence, since $\mbf{U} = \mbf{\Lambda}\mbf{A}$, we have that $\mbf{U} \in \mathscr{A}_v^1$.
\item
Suppose that $\mbf{A} \in \mathscr{C}_v$. By \cite{GrKl12}, we know that there is some $\gamma_{\mbf{B}} \in
\ell^1_v$ such that the entries of $\mbf{B}_{22}$ and  $\mbf{B}_{22}^{-1}$ are bounded by
$\gamma_{\mbf{B}}$ for all $n$. Then, by \eqref{B_estimate}, we have
\[ |\lambda_{kj}| \leq  \sum_{i = 1}^{k-1} |\beta_{ki} b_{ij}| \leq\sum_{i \in \mathbb{Z}} \gamma_{\mbf{B}}(k-i)\gamma_{\mbf{B}}(i-j) = [\gamma_{\mbf{B}}*\gamma_{\mbf{B}}](k-j). \]
Since $\ell^1_v$ is a convolution algebra, we have that $\mbf{\Lambda}$ satisfies the appropriate decay estimate, and so does $\mbf{L}$ by inverse-closedness.

Now consider $\mbf{L}_{11}$. Then, there is some sequence $\gamma_{\mbf{A}} \in \ell^1_v$ such that the entries of
$\mbf{A}_{11}$ and $ \mbf{A}_{11}^{-1}$ are bounded by $\gamma_{\mbf{A}}$ for all $n$.
Then, by \eqref{l_bound}, we have
\[ |\ell_{jk}| \leq \sum_{i = -1}^{1-k} |a_{ji}\alpha_{ik}| \leq \sum_{i \in \mathbb{Z}} \gamma_{\mbf{A}}(j-i)\gamma_{\mbf{A}}(i-k) = [\gamma_{\mbf{A}}*\gamma_{\mbf{A}}](j-k). \]
Again, the result follows from the fact that $\ell^1_v$ is a convolution algebra.
\end{enumerate}
\end{proof}

\section{Localization of the Cholesky, QR-, and Polar factorization} \label{s:cholesky}

\subsection{The Cholesky factorization}

\begin{cor} 
\label{cor:cholesky}
Assume that $\mathscr{A}$ is a strongly decomposable inverse closed subalgebra of $\mathscr{B}(\ell^2)$ that satisfies \eqref{incl}  and \eqref{dominorm} (for example we can choose  $\mathscr{A}$ 
to be one of the algebras in Definition \ref{algebras}). Then any positive definite matrix $\mbf{A}\in\mathscr{A}$  
admits a  Cholesky factorization in $\mathscr{A}$.
\end{cor}

\begin{proof}
Since $\bA$ is positive definite we can always rescale $\bA$ so that $\|\bI - \alpha\bA\|_{\Blt} < 1$ for some $\alpha > 0$ (note that we only claim $\|\bI - \alpha\bA\|_{\Blt} < 1$ and not $\|\bI - \alpha\bA\|_{\mathscr{A}} < 1$) . Hence,
by Theorem \ref{GLdec} we  know that $\mbf{A} = \mbf{L}\mbf{U}$ where $\mbf{L}, \mbf{U} \in \mathscr{A}$ and $\mbf{L}_{ii}=1$, $i\in \mathbb Z$. 
Since $\bA = \bA^* = \bL \bU = \bU^* \bL^*$, there holds $\bA = \bL \bD \bL^*$, where $\bD$ has to be positive
definite, because $\bA$ is. Then $\bA = (\bL \bD^{\frac{1}{2}}) (\bL\bD^{\frac{1}{2}})^*$
is the Cholesky factorization of $\bA$, and obviously $\bL \bD^{\frac{1}{2}} \in \mathscr{A}$.
\end{proof}


We note that for the special case when $\mathscr{A}$ is the Jaffard algebra and $\bA$ is finite-dimensional,
Corollary~\ref{cor:cholesky} recovers Lemma A.1 in~\cite{HallJin09}. Furthermore,  Corollary~\ref{cor:cholesky}
contains as special case Lemma II.3.2. of~\cite{GKW89}, by setting  $\mathscr{A}={\cal C}_v$ with $v \equiv 1$.

\subsection{The QR-factorization}
\begin{cor}
\label{cor:qr}
Assume that $\mathscr{A}$ is a strongly decomposable inverse closed subalgebra of $\mathscr{B}(\ell^2)$ that satisfies \eqref{incl}  and \eqref{dominorm}  (for example we can choose  $\mathscr{A}$ 
to be one of the algebras in Definition \ref{algebras}).  Let $\mbf{A} \in \mathscr{A}$ be such that $\mbf{A}$ is invertible and suppose that its QR-decomposition $\mbf{A} = \mbf{Q}\mbf{R}$ exists. Then $\mbf{Q},\mbf{R} \in \mathscr{A}$.
\end{cor}
\begin{proof}
Consider the matrix $\mbf{V} = \mbf{A}^*\mbf{A} = \mbf{R}\mbf{Q}^*\mbf{Q}\mbf{R} = \mbf{R}^*\mbf{R}$. Note that $\mbf{V} \in \mathscr{A}$ and it is invertible, positive definite, and hermitian. Let $\mbf{V} = \mbf{C}^*\mbf{C}$ be its Cholesky factorization. Then $\mbf{R} = \mbf{D}\mbf{C}$, where $\mbf{D}$ is a unitary diagonal matrix. So $\mbf{R} \in \mathscr{A}$, which implies $\mbf{R}^{-1} \in \mathscr{A}$. Since $\mathscr{A}$
is an algebra, it follows that $\mbf{Q} \in \mathscr{A}$.
\end{proof}

It is claimed (without proof) in Section II.B of~\cite{KM98} that the Gram-Schmidt orthogonalization would destroy 
the localization properties of a given set of vectors, while applying the L\"owdin orthogonalization procedure
(which essentially amounts to applying the inverse square root of a matrix) preserves localization. 
Corollary~\ref{cor:qr} shows that this claim is incorrect, since Gram-Schmidt does preserve localization as well.

\subsection{The Polar factorization}
\begin{thm}
Let $\mathscr{A}$ be one of the decay algebras in Definition \ref{algebras}. Let $\mbf{A} \in \mathscr{A}$ be invertible and suppose its polar decomposition $\mbf{A} = \mbf{P}\mbf{U}$ exists, where $\mbf{P}$ is positive definite. Then $\mbf{P}, \mbf{U} \in \mathscr{A}$.
\end{thm}
\begin{proof}
Recall that $\mbf{P} = \left( \mbf{A}^*\mbf{A} \right)^{-\frac{1}{2}}$. Then, the Banach square root
theorem~\cite{Gar66} implies that $\mbf{P} \in \mathscr{A}$. So, since $\mathscr{A}$ is inverse-closed, we have
$\mbf{P}^{-1} \in \mathscr{A}$, which in turn implies $\mbf{U} \in \mathscr{A}$. 
\end{proof}

A close relative of the polar factorization is the singular value decomposition 
$\mbf{A} = \mbf{V}\mbf{\Sigma}\mbf{W}^*$.
Since $\mbf{VW}^* = \mbf{U}$, where $\mbf{U}$ is the partial isometry from the polar decomposition of $\mbf{A}$, 
we can say that $\mbf{VW}^*$ inherits the decay of $\mbf{A}$. Of course, this says nothing about the
localization of $\bV$ and $\bW$. Indeed, as already mentioned in the introduction, it is clear that without 
any additional assumptions the individual 
factors $\mbf{V}$ and $\mbf{W}$ cannot inherit the decay properties of $\mbf{A}$. For example, note that the 
identity matrix $\mbf{I}$ can be written $\mbf{I} = \mbf{VI}\mbf{V}^*$ for any arbitrary unitary matrix $\mbf{V}$. 
Eigenvector localization is a heavily studied subject, in particular in connection with the famous phenomenon of  
Anderson localization (see e.g.~\cite{Sto11} and its many references). We will report on eigenvector localization
in connection with Banach algebras with off-diagonal decay in a forthcoming paper.


\section{Factorizations, Localization, and Functional Calculi} 
\label{s:functional}

The question whether matrix functions, such as the matrix exponential, inherit the localization properties
of a matrix has been investigated for the cases of band matrices in~\cite{BG99,Iserles00,BR07}. 
This issue is of relevance in various areas of numerical mathematics~\cite{BG99,Iserles00}. What
can we say about functions of a matrix if the matrix is not banded, but has some form of off-diagonal decay?

In the previous sections we were able to prove that localization of various matrix factors is preserved in general
under reasonably mild conditions and regardless of how these factors were obtained. In practice, the factors are often obtained via some kind of functional calculus. Since many useful functional calculi restrict naturally to the inverse closed subalgebras, the factors will automatically inherit the property.

For example, in case of Laurent matrices  a necessary and sufficient condition for factorization of a matrix  $\mbf{A} = (a_{jk}) \in \mathscr{C}_1$ is given in terms of its symbol 
$\sigma_{\mbf{A}} \in C(\mathbb{T})$, $\sigma_{\mbf{A}} (\theta)= \sum\limits_{n\in\mathbb Z} c_n \theta^{n}$,
$\theta\in\mathbb T$, $c_n = a_{j,j-n}$, $j,k,n\in\mathbb Z$.
Recall that in this case $\mbf{A}$ is the matrix of an operator of multiplication by $\sigma_{\mbf{A}}$ in
$L^2(\mathbb T)$ in the standard Fourier basis. Thus, the LU-factorization of $\mbf A$ is
equivalent to the \emph{spectral} factorization $\sigma_{\mbf A} = \sigma_{\mbf L}\sigma_{\mbf U}$,
where $\sigma_{\mbf L}$ admits a non-vanishing continuous extension that is holomorphic in 
$\mathbb D$ and $\sigma_{\mbf U}$ admits a non-vanishing continuous extension that is holomorphic in $\mathbb C\backslash\overline{\mathbb D}$.
The factorization condition then is the well-known Paley-Wiener condition \cite{JP01} 
\[\int_{-\pi}^{\pi} \ln \sigma_{\mbf{A}}(e^{it})dt > -\infty,\]
and the (extension of the) symbol of one of the factors, $\sigma_{\mbf{L}}$, is then computed via
\[\sigma_{\mbf{L}}(z) = \exp\left(\frac1{4\pi}\int_{-\pi}^{\pi} \ln \sigma_{\mbf{A}}(e^{it})\frac{e^{it}+z}{e^{it}-z}dt\right),\quad z\in\mathbb D. 
\]
Since the series defining the matrix exponential will converge in any Banach algebra of matrices, the factors will belong to a subalgebra $\mathscr A$ as long as the logarithm of the symbol of the matrix stays in the isomorphic subalgebra of $C(\mathbb T)$. We cite \cite{JP01,BP07} 
for more information and detailed estimates on the coefficients of the spectral factors.

Riesz-Dunford functional calculus is another natural example of a calculus that restricts to a Banach subalgebra. Indeed, if $\mbf A\in \mathscr A$ and $\mathscr A$ is inverse closed then 
\[\mbf B = \frac1{2\pi i}\int_\Gamma f(\lambda)(\lambda\mbf I -\mbf A)^{-1}d\lambda \in\mathscr A,\] 
where $f$ is holomorphic inside a positively oriented contour $\Gamma$ surrounding the spectrum
or a spectral component of $\mbf A$. To give a concrete example, it  now follows easily that 
the matrix exponential $\exp(\mbf{A}) \in \mathscr{A}$ if $\bA \in \mathscr{A}$, whenever
$\mathscr{A}$ is an inverse-closed Banach algebra, such as e.g.\ one of the Banach algebras
in Definition~\ref{algebras}.


Yet another example of a useful functional calculus is in some sense a generalization of the approach in the Laurent case. Given $\mbf A\in \mathscr B_c$ and $h\in L^1(\mathbb R)$, we can define a Banach
$L^1(\mathbb R)$-module structure \cite{BK05} via
\[h\mbf A = \int_{\mathbb T} h(\theta)f_{\mbf A}(\theta^{-1})d\theta = \int_{\mathbb T} h(\theta)M(\theta^{-1}){\mbf A}M(\theta)d\theta.\]
This structure extends to a closed operator calculus as in \cite{BK11}.
Again,  we have $h\mbf A \in \mathscr A$ as long as $\mbf A$ is in a Banach algebra $\mathscr A$ that is invariant under the modulation representation.


\subsection*{Acknowledgements}
Thomas Strohmer  acknowledges partial support from the NSF via grant
DTRA-DMS 1042939, and from DARPA via grant N66001-11-1-4090.
Tim Wertz was partially supported by  NSF via grant DTRA-DMS 1042939.

\def\cprime{$'$}

\end{document}